\documentclass[18pt]{amsart}
\usepackage{hyperref}
\usepackage[usenames]{color}
\usepackage{amsmath,amsthm,amsfonts,amssymb}

\newtheorem{theorem}{Theorem}[section]
\newtheorem{lemma}[theorem]{Lemma}

\newtheorem{proposition}[theorem]{Proposition}
\newtheorem{problem}[theorem]{Problem}

\theoremstyle{definition}
\newtheorem{definition}[theorem]{Definition}

\theoremstyle{remark}

\numberwithin{equation}{section}

\newcommand{\vr}{\noindent \color{red}}

\begin{document}

\title{On the Andrews-Curtis groups: non-finite presentability}\footnote{The research was   funded by Sobolev Institute of Mathematics,  project  FWNF-22-0003.}

\author{Vitaly Roman'kov}

%\date{23 June 2012}

%\dedicatory{}

\begin{abstract}
The Andrews-Curtis conjecture  remains one of the outstanding open problems in combinatorial group theory. It claims that every normally generating $r$-tuple of a free group $F_r$ of rank $r\geq 2$ can be reduced to a basis by means of Nielsen transformations and arbitrary conjugations. These transformations generate the so-called Andrews-Curtis group AC($F_r$). The groups AC($F_r$) ($r = 2, 3, \ldots$) are actively investigated and allows various generalizations, for which there are a number of results. At the same time, almost nothing is known about the structure and properties of the original groups AC($F_r$). In this paper we define a class $\{A_{r, s}: r, s \geq 1\}$ of generalized Andrews-Curtis groups in which any group $A_{r,r}$ is isomorphic to  the Andrews-Curtis group  AC($F_r$). 
We prove that  every  group $A_{2,s}$\, ($s \geq 1$) is non-finitely presented.
Hence the Andrews-Curtis group AC($F_2$) $\simeq A_{2,2}$  is non-finitely presented.  Thus, we give a partial answer to the well-known question about the finite presentability of the groups AC($F_r$), explicitly stated by J. Swan and A. Lisitsa in the Kourovka notebook \cite{KN} 
(Question 18.89). 
\end{abstract}

\maketitle

%{\small  \tableofcontents}

\section{Introduction}

\subsection{Andrews-Curtis groups}
Further $F_r$ means a free group of rank $r$. In this paper we prove that the Andrews-Curtis group  AC($F_2$)  is non-finitely presented. In fact, we first  show that  the widely considered general Andrews-Curtis groups GAC($F_r$) are isomorphic to the corresponding groups AC($F_r$). Second, we prove that any Andrews-Curtis group AC($F_r$) is isomorphic to a member $A_{r,r}$ of an infinite series of finitely generated groups $\{A_{r,s}: r, s \geq 1\}$. At last, we prove that   each of groups $A_{2, s}$  is  non-finitely presented. It follows that AC($F_2$) $\simeq A_{2,2}$ is non-finitely presented.  We also discuss some interesting properties of Andrews-Curtis groups  AC$(F_r)$,\, $r \geq 2$, and related open problems.

The Andrews-Curtis groups AC($F_r$) for $r\geq 2$ were introduced  in   their connection with the famous  Andrews-Curtis Conjecture (ACC)   named after James J. Andrews and Morton L. Curtis who proposed it in 1965 \cite{AC}. The ACC claims that every balanced presentation of the trivial group can be reduced to the standard one by a finite sequence of ``elementary AC-transformations", which are Nielsen transformations augmented by conjugations. This problem is of interest in topology as well as in group theory. A topological interpretation of this conjecture was given in the original paper by Andrews and Curtis \cite{AC}.

See also \cite{Met}. For any $r\geq 2$, the AC-transformations generate the group AC($F_r$), which acts on the set NGen($F_r$) of all normally generating $r$-tuples $u=(u_1, \ldots , u_r)\in F_r^r$. AC-transformations also act on the set $F_r^r$ of all $r$-tuples, where they generate the general Andrews-Curtis group GAC($F_r$).   

These definitions generalize naturally from the group $F_r$ to any  group $G$.  

To explain, fix a natural number $r \geq 2$. The following are \emph{elementary Andrews-Curtis transformations} (or \emph{elementary AC-moves}) on the set $G^r$ of $r$-tuples of $G$:
 \begin{itemize}
 \item  $R_{ij}^{\pm 1}$: $(u_1, \ldots, u_i, \ldots,  u_r) \mapsto (u_1,
\ldots,u_iu_j^{\pm 1}, \ldots,  u_r), \, i \ne j;$
  \item  $L_{ij}^{\pm 1}$: $(u_1, \ldots, u_i, \ldots, u_r)   \mapsto  (u_1,
\ldots,u_j^{\pm 1} u_i, \ldots,  u_r),\   i \ne j;$
 \item $I_{i}$\,\,\,\,\,\,:  $(u_1, \ldots, u_i, \ldots, u_r)  \mapsto  (u_1,
\ldots,u_i^{-1}, \ldots, u_r);$
 \item $C_{i,w}$: $(u_1, \ldots, u_i,
\ldots, u_r) \mapsto (u_1,\ldots,u_i^{w}, \ldots, u_r ), \
w\in G.$
 \end{itemize}
 \noindent (For $g, f \in G$\, $g^f$ means $fgf^{-1}$.) 
 
Transformations   $R_{ij}^{\pm 1}, L_{ij}^{\pm 1}, I_i$ are  called \emph{elementary Nielsen transformations}. Composition of a  finite sequence of elementary Andrews-Curtis moves is called  an  {\it AC-transformation}.  Each inverse of an  AC-transformation is AC-transformation.  
 
Every such AC-move gives a bijection on the set $G^r$ of all $r$-tuples  of elements of $G$, which is an element of the symmetric group $Sym(G^r)$. The subgroup generated by all the AC-moves in $Sym(G^r)$ is called the \emph{general AC-group of $G$ of dimension $r$} and is denoted by $GAC_r(G)$. The set $NGen_r(G)$ of all $r$-tuples in $G^r$ that generate $G$ as a normal subgroup, is invariant under  all AC-moves, so every AC-move gives a bijection of the set $NGen_r(G)$. The subgroup generated by all the AC-moves in $Sym$(NGen$_r$($G$)) is called the \emph{AC-group of $G$ of dimension $r$} and is denoted by AC$_r$($G$).  The group AC$_r$($F_r$) is the true Andrews-Curtis group of $F_r$  directly related to ACC in $F_r$. However, the set NGen$_r$($F_r$) may have a much more  complicated structure than the set $F_r^r$, in this regard it would be easier to study the group GAC$_r$($F_r$). Observe, that the restriction of an element $\alpha \in$  GAC$_r$($F_r$), viewed as a bijection on $F_r^r$, gives  a bijection $\bar{\alpha}$ on NGen$_r$($F_r$), which is an element of AC$_r$($F_r$). It is easy to see that the map $\alpha \to \bar{\alpha}$ gives a homomorphism $\phi: $ GAC$_r$($F_r$) $\to$ AC$_r$($F_r$), which is onto. Of course, we are interested only in the case when the set $NGen_r(G)$ is not empty.

 Note that the set Gen$_r$($G$) of generating $r$-tuples of $G$
 is closed under Nielsen moves, while the set NGen$_r$($G$) of normal-generating $r$-tuples of $G$ is closed under AC-moves.

Two tuples $u, v \in G^r$  are called  \emph{AC-equivalent}  (symbolically  $u \sim_{AC} v$) if there is  an AC-transformation which moves $u$ to $v$. 

If $G$ is finitely generated, then it suffices to have a finite number of  elementary AC-moves, where the element $w$ in $C_{i,w}$ or its inverse $w^{-1}$ belongs to a fixed finite set generating $G$.

\medskip
\noindent
\subsection{Andrews-Curtis Conjecture}  
{\it Let $F_r$ be a free group with basis $f = (f_1, \ldots,f_r)$. Then  a tuple of elements
  $u = (u_1, \ldots, u_r)$ generates $F_r$ as a normal subgroup if and only if
   $u\sim_{AC} f$.}

\medskip

To state ACC in  the original form recall that a presentation is called {\it balanced} if the  number of generators is equal to the number of relators.  Furthermore, a balanced presentation  $\langle f \mid u\rangle$, where  $f = (f_1, \ldots, f_r), u =  (u_1, \ldots, u_r)$  is called \emph{AC-trivializable} if  $u \sim_{AC} f$. The original ACC states that every balanced presentation of the trivial group is trivializable.

In 1985 Akbulut and Kirby \cite{AK} came up with a series of 
``potential counterexamples" to ACC:
 $$
AK(n) =  \langle x,y| ~x^n = y^{n+1}, ~xyx=yxy\rangle, ~n \ge 2.
 $$
Presentations $AK(n)$ are balanced  presentations of the trivial group.  Akbulut and Kirby   conjectured that they are not trivializable, i.e., 
that the pair of its relators $(x^ny^{-n-1},xyxy^{-1}x^{-1}y^{-1})$ is not AC-equivalent to the pair of generators $(x,y)$.  It turned out later that the presentation $AK(2)$  is AC-trivializable (see {\vr \cite{MM}}), so $AK(2)$ is not a counterexample to ACC.  The question whether or not the presentations $AK(n)$ with $n >2$ are trivializable is still open despite an ongoing effort by the research community (see \cite{MMS}). Note that currently $AK(3)$ is the shortest (in the total length of relators) potential counterexample to ACC.  Indeed, 
in \cite{HR} Havas and Ramsay proved that if $\langle x,y \mid  u= 1, v = 1 \rangle$ is a presentation of the trivial group with  $|u| + |v| \leq 13$ then either
$(u,v) \sim_{AC} (x,y)$
or
$(u,v) \sim_{AC} (x^3y^{-4}, xyxy^{-1}x^{-1}y^{-1})$.

 Other infinite series of potential counterexamples are given in \cite{BM}, \cite{MS} and in \cite{MMS}.  Finally, we mention a positive solution of a similar problem for free solvable groups  by A. Myasnikov \cite{M}. See papers \cite{BM} and \cite{MMS} for more details and some particular results. 

Currently there are  several approaches to resolve  ACC. One is based on the following observation: if $r$-tuples $u$ and $v$ in a free group $F_r$ are AC-equivalent then for any group homomorphism $\phi:F_r \to G$ one has $\phi(u) \sim_{AC} \phi(v)$ in $G$.  Therefore, to show, for example, that $AK(3)$ is a counterexample to ACC it suffices to find a group $G$ and a pair of elements $a, b \in G$ such that $(a,b)  \not \sim_{AC} (a^3b^{-4}, abab^{-1}a^{-1}b^{-1})$ in $G$. This brought up an interesting research on Andrews-Curtis Conjecture in arbitrary groups $G$. Observe, that ACC can be easily reformulated in a direct way for an arbitrary group $G$, by saying that any two $r$-tuples of elements from $G$ that generate $G$ as a normal subgroup  are AC-equivalent in $G$. However, this straightforward version of ACC immediately fails for some  groups, in particular,  for some finitely generated abelian groups with torsion (see \cite{BKM}), and as a consequence,  for some groups $G$ with torsion in their abelianization. Fortunately, one can slightly adjust the formulation  of ACC to accommodate   torsion in the abelianization of $G$. It turns out that the resulting version of ACC, termed the \emph{generalized Andrews-Curtis Conjecture}, holds in nilpotent and solvable groups \cite{M}, \cite{G}, finite groups \cite{BLM}, and some other groups. Up to now, this approach provides more and more groups satisfying the generalized ACC, and no counterexamples to the original one.

Another approach calls to study the group structure of AC-transformations  and apply this knowledge to ACC. Almost nothing is known in this direction and our goal here is to present some results on algebraic structure of the group of AC-transformations of  NGen$_r$($F_r$).

The group AC$_r$($F_r$) for $r\geq 2$ is the classical AC-group that was introduced in relation with ACC.  In Section \ref{sec:2} we prove the following result on AC-groups. 

\medskip
\noindent
{\bf Theorem A.}
\label{th:A}
 {\it For any $r\geq 2$ the homomorphism $\phi:$ GAC$_r$($F_r$) $\to$ AC($F_r$) is an isomorphism.
}

\medskip
In  Section \ref{sec:3}  we define a class $\{A_{r, s}: r, s \geq 1\}$ of generalized Andrews-Curtis groups in which any group $A_{r,r}$ is isomorphic to  the Andrews-Curtis group  AC($F_r$). Section \ref{sec:4} is devoted to preliminary results. In Section \ref{sec:5} the following main results are proved. 

\medskip
\noindent
{\bf Theorem B.}
\label{th:B}  {\it The groups $A_{2,s}, s \geq 1,$ are not finitely presentable. }

Since $AC(F_2)$ is antiisomorphic to $A_{2,2}$ we get the  following result that solves the 
Problem 18.89 posed  by J. Swan and A.P.  Lisitsa in the Kourovka Notebook \cite{KN}. 

{\bf Corollary C.}  {\it The group $AC(F_2)$ is not finitely presentable.
}

Section \ref{sec:6} contains open problems proposed in collaboration with  A.G. Myasnikov. 

\medskip

\section{AC-groups and automorphisms of free groups }
\label{sec:2}

There is a deep and interesting relationship between the AC-groups $AC(F_r)$ and automorphisms of free groups.

Let $G$ be a group, and Aut($G$) the automorphism group of $G.$ Denote by IAut($G$) the subgroup of Aut($G$) consisting of those automorphisms of $G$ which induce the identity map on the commutator quotient group ({\it  abelianization}) $G/G'.$ 

As above $F_r$ is  a free group of rank $r$ with basis $f = (f_1, ..., f_r)$.  
 J. Nielsen, in 1924 (compare \cite{Niel}, \cite{MKS}), using hyperbolic geometry gave finite sets of generators and relations for the automorphism group Aut($F_r$).  Specifically Aut($F_r$) is generated by automorphisms of the following three forms (all indexes $i, j, l$ range over the set $\{1, ..., r\}$ subject only the condition  that $i\not= j,$ every automorphism  changes exactly one free generator):
\begin{itemize}
\item $\rho_{i,j} : f_i \mapsto f_if_j, f_l \mapsto f_l $ for $l \not= i,$
\item $ \lambda_{i,j}: f_i \mapsto f_jf_i, f_l \mapsto f_l$  for $l\not= i,$ 
\item $\iota_i\, \,\,\,\,: f_i \mapsto f_i^{-1}, f_l\mapsto f_l$ for $l\not= i.$
\end{itemize}
  It was generally agreed  that both Nielsen's methods and his results were complicated (see \cite{MKS}, p. 164), and the subject did not progress for long time. In 1974 J. McCool \cite{McCool} has found  a new finite presentation for Aut($F_r$). McCool's generators were the Whitehead's automorphisms and his relations were of two types: the relations among the integral monomial $r \times r$ matrices and relations he called R1 - R6 according to two types of the Whitehead's automorphisms \cite{Whi}. See details in \cite{LS}.

 It was  shown by W. Magnus \cite{Mag34}, using a work of J. Nielsen \cite{Niel1}, that IAut($F_r$) has a finite generating set of the following two forms (all indexes $i, j, l$ range over the set $\{1, ..., r\}$ subject only the condition  that $i\not= j, k; j\not= k$ every automorphism  changes exactly one free generator):
\begin{itemize}
\item $\xi_{i,j}: f_i \mapsto f_i^{f_j}, f_l \mapsto f_l $ for $l \not= i,$
\item $\rho_{i,j,k}: f_i \mapsto f_i[f_j, f_k], f_l \mapsto f_l $ for $l \not= i$
\end{itemize}
 Here $[g,h]$ denotes  commutator $ghg^{-1}h^{-1}$ of elements $g, h.$ The statement includes the result of J. Nielsen \cite{Niel} that IAut($F_2$) = Inn($F_2$), the subgroup of inner automorphisms of $F_2.$

Since the group IAut($F_2$) is isomorphic to Inn($F_2$) and so to $F_2,$ it is  finitely  presented. W. Magnus \cite{Mag34} raised the question on the finite presentability of IAut$(F_r)$ for $r \geq 3.$  
The question attracted the attention of researchers, among them O. Chein \cite{Chein}, S. Bachmuth \cite{Bach}, J. Smillie and K. Vogtman \cite{SV}.

The first result in solving this problem has been obtained by S. Krsti\'c and J. McCool.

\begin{theorem}
\label{th:1.1}
(\cite{KrMc}.) Let $G=F_3/N$, where $N$ is a characteristic subgroup of $F_3$ such that $N$ is contained in the second derived subgroup $F_3''$
of $F_3.$ Let $\bar{\xi}_{1,2}, \bar{\xi}_{1,3}, \bar{\rho}_{1,2,3}$ be the elements of Aut($G$) induced by $\xi_{1,2}, \xi_{1,3}, \rho_{1,2,3}$ respectively. Then any subgroup of IAut($G$) which contains the set $\{\bar{\xi}_{1,2}, \bar{\xi}_{1,3}, \bar{\rho}_{1,2,3}\}$ is not finitely presentable. 
\end{theorem}

This statement covers the cases of the free group $F_3,$ of the free metabelian group $M_3 = F_3({\mathcal A}^2),$ and more generally every relatively free group 
$F_3( {\mathcal C})$ of rank $3$ where ${\mathcal C}$ is a variety of groups containing the variety ${\mathcal A}^2$ of all metabelian groups.

Let $F_{r,s}$ be a free group of finite rank $r+s$ with a set of free generators $f \cup y$, where $f = \{f_1, \ldots , f_r\}, y = \{y_1, \ldots , y_s\}$. We assume that $r \geq 1$  and $s\geq 0$. The group $\tilde{A}_{r,s}$ is defined as a subgroup of the automorphism group Aut($F_{r,s}$) generated by the automorphism group Aut($F_r$), more precisely, the subgroup of all automorphisms fixing the generators $y$ in  $F_{r,s}$, and mapping the subgroup generated by the elements of $f$ (isomorphic to the group $F_r$) onto itself,
together with automorphisms $\xi_{i,k}$ for $i = 1, ..., r$ and $k = 1, ..., s,$ that are defined on the basic elements as follows:
\begin{equation}
\xi_{i,k} : f_i \mapsto f_i^{y_k},   f_l \mapsto f_l \  \textrm{and}\    y_t \mapsto y_t, \ \textrm{for} \  i, l = 1, ... , r; \  l \not= i,  \  
  \textrm{and} \   k, t = 1, ..., s.
\end{equation}
Thus the group $\tilde{A}_{r,s}$ is isomorphic to a subgroup of the group Aut($F_{r,s}$).  
Then $\tilde{A}_{r,s}$ acts on its orbit $O_{r,s}$ generated by all tuples of the form $(u_1, \ldots , u_r, y_1, \ldots , y_s),$, where $u = ( u_1, \ldots , u_r) \in F_r^r$. Denote by $O_r$ the set of corresponding truncated tuples $u$. Then $\tilde{A}(r, s)$ also acts on the set $O_r$. Of course, $\tilde{A}_{r, s}$ also acts on the set $F_{r,s}^r\supseteqq O_r.$ 

Let us define the AC-transformation group $A_{r,s}$, which naturally corresponds to the group $\tilde{A}_{r,s}$.
To explain we need a more detailed notation of AC-moves in $F_{r,s}$. 
For any tuple of the form $u = (u_1, \ldots , u_r, y_1, \ldots , y_s)\in F_{r,s}^{r+s}$ we put (all moves change exactly one component of the tuple):
\begin{itemize}\item AC$_1$($i, j$) \    \   replaces $u_i$ by $u_iu_j,$ \  where $1 \leq i, j \leq r, i\not=j.$
\item AC$_2$($i,j$) \      \  replaces $u_i$ by $u_ju_i,$\  where $1 \leq i, j \leq r, i\not=j.$
\item AC$_3$($i$) \ \ replaces $u_i$ by $u_i^{-1}$, \  where $1 \leq i \leq r.$
\item AC$_4$($i,k$) \  replaces $u_i$ by $u_i^{y_k},$  \  where $1 \leq i \leq r\, 1 \leq k \leq s.$
\end{itemize}
By definition, the group $A_{r,s}$ is generated by these  transformations. Then this group also acts on the orbit $O_{r,s}$, on $O_r$ and on $F_{r,s}^r$.  
   It is well known that   $A_{r,s}$  is antiisomorphic to $\tilde{A}_{r,s}$ via the following map:
 \begin{itemize}
 \item AC$_1(i,j) \mapsto \rho_{i,j},$
 \item AC$_2(i,j) \mapsto \lambda_{i,j},$
 \item AC$_3(i) \mapsto \iota_i,$
 \item AC$_4(i,k) \mapsto \xi_{i,k}.$
 \end{itemize}  
The proof one can find in the monograph in the monograph \cite{MKS} (p. 130) or in the monograph \cite{LS} (section ``Automorphisms of free groups'').

\section{Proof of Theorem A}
\label{sec:3}

Let $s = r$ and $A_r$ denote the restriction of $A_{r,r}$ to  $O_r.$
Theorem A is obtained as a consequence of the following assertions, which have independent significance.

\begin{lemma}
\label{le:2.1}
For any $\alpha \in A_r$, if $\alpha (x) = x$, where $x = (x_1, \ldots , x_r)$ is a basis of $F_r$, then $\alpha$ acts identically to $F_{r,s}^r$, in particular to $O_r$.   Moreover, if we change in each transformation of the form AC$_4$($i, k$)  the conjugator $y_k$ by $x_k$  in $\alpha$ and denote the resulting transformation by $\beta $, then $\beta $ is identical on $F_r^r.$ Hence $\nu :\alpha \mapsto \beta$ is a homomorphism $A_r \rightarrow $ GAC$_r$($F_r$).
\end{lemma}
\begin{proof}
Let $v = (v_1, \ldots , v_r) \in F_{r,s}^r$. We define a homomorphism $\varphi : F_r^r \rightarrow 
F_{r,s}^r, \, x_i \mapsto v_i \, (i = 1, \ldots , r).$ Then 
$$\varphi (\alpha (x)) = \alpha (\varphi (v)) = \alpha (v) = \varphi (x) = v.$$

Moreover, since $\alpha (x) = x$ all $y_k$ in $\alpha$ cancel among themselves. Then  the elements $x_k$ corresponding to $y_k$ in $\beta$  cancel among  themselves too. Define homomorphism $\mu : F_{r,r} \rightarrow F_r$ such that $\varphi (x_i) = v_i, \varphi (y_k) = x_k$ for $i, k = 1, \ldots , r.$ 
 Then $\mu (\alpha (x)) =  \alpha (\mu (x)) = \beta (v)  = v. $
 Then $\beta$ is  identical as element of GAC$_r$($F_r$) and $\nu $ is a homomorphism.  
\end{proof}

\begin{lemma}
\label{le:2.2}
If $\alpha \in A_r$ is not identical on $O_r$ then the transformation $\beta\in A_{r,r}$ derived from $\alpha$ as in Lemma \ref{le:2.1}    is not identical on  $F_r^r$. Moreover, $\beta$ is not identical on NGen$_r$($F_r$).  
It follows that  $A_{r,r}$ is isomorphic to  GAC$_r$($F_r$) and AC($F_r$). Therefore, GAC$_r$($F_r$) $\simeq$ AC($F_r$). 
\end{lemma}
\begin{proof}
  Since $\alpha$ is not identical on $O_r$, then by lemma \ref{le:2.1} $\alpha$ is not identical on $F_r^r$. Then  $\alpha (f) \not = f.$ Suppose that   all occurences of $y_1, \ldots , y_r$ in formal expression of $\alpha$ cancel among themselves. Then all corresponding occurences of $f_1, \ldots , f_r$ in $\beta = \nu (\alpha )$ also cancel among themselves. Then $\beta (f) \not= f$ and $\beta$ is non-trivial. 
  
Now suppose that not all occurences of  $y_1, \ldots , y_r$ in formal expression of $\alpha$ cancel among themselves and $\beta (f) = f$.   
  Suppose that $\alpha$ includes $m$  occurences of AC$_4$-transformations or their inverses. Denote  $u = (u_1, \ldots , u_r)$, where  $u_1=f_2^{2m}f_1f_2^{2m},u_2=f_3^{2m}f_2f_3^{2m}, \ldots , u_{r-1} = f_r^{2m}f_{r-1}f_r^{2m}, 
u_r=f_1^{2m}f_rf_1^{2m}$.  Then elements $u$ generate a free subgroup of rank $r.$ Obviously, the reduction process for components of elements coming from AC$_4$-transformations of $\alpha (u)$ does not cancel more than $m$ letters from a side of  any entry of $u_i$ for every $i.$  It follows that $\alpha (u)\not= u.$  Note that $u \in $ NGen($F_r$). Thus $\beta$ is non-trivial as an element of AC($F_r$). Therefore $A_r \simeq $ GAC($F_r)\simeq $ AC($F_r$).  
\end{proof}  

Therefore Theorem A is proved.

\section{Auxiliary  results}
\label{sec:4}

\subsection{Essentially infinite sets of relations}

In this Section we are following to \cite{KrMc}. Let $F(X)$ be the free group with basis $X,$  and let $\mu : F(X) \rightarrow G$ be a group epimorphism. If $S$ is a subset of $X,$  then
elements of ker$(\mu ) \cap $ gp($S$) are called {\it relations of $G$ on the generators $S$}. If $Q$ is a set of such relations, we say that $Q$ is {\it essentially infinite} if there is no finite subset $W$ of ker($\mu $) such that $Q$ is contained in the normal closure ncl($W$) of $W$ in $F(X).$

It was shown in \cite{KrMc} that this notion is independent of the choice of the generating set $X$ of $X,$ containing $S.$ Also, it was noted that if $Q$ is an essentially infinite set of relations of $G$ on the finite set $S,$ then no subgroup $H$ of $G$ containing the image $\mu (S)$ can be finitely presentable. Indeed, if $H$ is finitely presentable, then there is a finite preimage in $F(X)$, containing $S$,  of a generating set of $H$.  Since $H$ is finitely presentable on this set of generating elements, it would follow that $Q$ is contained in the norml closure of a finite set of relations of $G$, which is a contradiction.  

Furthemore, it was proved in \cite{KrMc} that if we have two free groups $F(X_1)$ and $F(X_2)$,  two subsets $S_1\subseteq X_1$ and $S_2\subseteq X_2$ and

\begin{equation}
\begin{array}{ccc}
 \psi' : F(S_1) & \rightarrow  & F(S_2)\\
 \downarrow  &&\downarrow  \\
\psi : G_1 & \rightarrow & G_2\\
\end{array}
\end{equation}

\noindent
be a commutative diagram of groups,    $Q \subseteq F(S_1) $ be a set of relations of $G_1,$ $Q_2 = \psi' (Q_1)$ be an essential infinite set of relations of $G_2,$ then $Q_1$ is an essential infinite set of relations of $G_1.$

\subsection{Fox derivatives}

For a given positive integer $r,$ we consider the free  group $F_r$ with basis  $\{f_1, ..., f_r\}$ and the integral group ring $\mathbb{Z}F_r.$ 
We use the partial derivatives introduced by Fox \cite{Fox}. An exellent introduction to the theory of the Fox derivatives and possible applications of them can be found in \cite{CF} 
 (see also \cite{Romess} and \cite{Timbook}). In our notation, these are defined as follows.

For $j = 1, ..., r,$ the (left) Fox derivative associated with $f_j$ is the linear map $D_j :\mathbb{Z}F_r \rightarrow \mathbb{Z}F_r$ satisfying the conditions

\begin{equation}
\label{eq:der1}
 D_j(f_j) = 1, D_j(f_i) = 0 \   \textrm{for} \     i \not= j
\end{equation}

\noindent
and

\begin{equation}
 \label{eq:der2}
D_j(uv) = D_j(u) + uD_j(v) \  \textrm{for all} \  u, v \in F_r.
\end{equation}

Obviously, an element $u \in F_r$ is trivial if and only if $D_i(u) = 0$ for all $i = 1, ..., r.$ Also note that for an arbitrary element  $g $ of $ F_n$ and every $j = 1, ..., n,$ $D_j(g^{-1}) = -g^{-1}D_j(g).$

The  {\it  trivialization} homomorphism  $\varepsilon : \mathbb{Z}F_r \rightarrow \mathbb{Z}$   is defined on the generators of $F_n$ by $\varepsilon (f_i) = 1$ for all $i = 1, ..., r$
and extended linearly to the group ring $\mathbb{Z}F_n.$

The  Fox derivatives appear in another setting as well. Let $\Delta F_r$ denote the fundamental ideal of the group ring $\mathbb{Z}F_r.$
It is a free left $\mathbb{Z}F_r-$module with a free basis consisting of $\{f_1-1, ..., f_r-1\}.$ This it leads us  to
the following formula which is called the {\it main identity} for the Fox derivatives:

\begin{equation}
 \label{eq:mainf}
\sum_{i=1}^{r} D_i(\alpha )(f_i - 1) = \alpha - \varepsilon (\alpha ) ,
\end{equation}

\noindent
where $\alpha \in \mathbb{Z}F_r.$ Conversely, if for any element $f \in  F_r$ and $\alpha_i  \in \mathbb{Z}F_r$ we have equality

\begin{equation}
 \sum_{i=1}^{r} \alpha_i (f_i - 1) = f - 1,
\end{equation}

\noindent
then $D_i(f) = \alpha_i$ for $i = 1, ..., r.$

More generally, we call a linear map $D: \mathbb{Z}F_r \rightarrow \mathbb{Z}F_r$
the {\it  free Fox derivative}  if $D$ satisfies the property

\begin{equation}
 D(uv) = D(u) + uD(v)
\end{equation}

\noindent
for all $u, v \in F_r.$ Every such derivative has the form

\begin{equation}
 D = \alpha_1 D_1 + ... + \alpha_n D_r,
\end{equation}

\noindent
where $\alpha_i = D(f_i)$ for $i = 1, ..., r.$ By definition $(\alpha D) (u) =  D(u)\alpha $ for any $\alpha \in \mathbb{Z}F_r, u \in F_r.$
Conversely, we can define a derivative $D = \sum_{i=1}^r \alpha_iD_i$ for arbitrary tuple of elements $\alpha_i \in \mathbb{Z}F_r.$

The notion of the  free Fox derivatives, defined above for free  groups, can be generalized  to groups 
of the type  $F_r/R',$  where $R$ is any normal subgroup of $F_r.$  First we give some notation.

Let $F_r$ be a free group with basis $\{f_1, ..., f_r\}.$ Let  $R$ be a normal subgroup of $F_r$ and that 
$R'$ is its derived subgroup (it is also normal subgroup of $F_r).$  Let $\bar{G} = F_r/R',$  $G = F_r/R$
and let $\mathbb{Z}G$ be  the group ring of $G.$ By $\mu : F_r  \rightarrow G$ we denote the standard epimorphism, as well as its linear extension to $\mu : \mathbb{Z}F_r \rightarrow \mathbb{Z}G.$  By $\bar{\mu} : \mathbb{Z}\bar{G} \rightarrow \mathbb{Z}G$ we mean the epimorphism induced by $\mu $.  Also we denote by $\mu ': F_r \rightarrow \bar{G}$ the other standard epimorphism. Obviously $\mu $ is a superposition of $\mu '$ and $\bar{\mu }.$

An abelian normal subgroup $\bar{R} = R/R'$ of the group $\bar{G}$ is considered as a module over the group ring $\mathbb{Z}G$
where action of $\mathbb{Z}G$ on $\bar{R}$ is induced by conjugation in $\bar{G}.$ This module is called the {\it relation module} of
$\bar{G}.$

Every free Fox partial derivative $D_j$ induces    a linear map $d_j : \mathbb{Z}F_r \rightarrow \mathbb{Z}G$ via $\mu .$ These linear maps $d_j$ also are called the {\it free Fox derivatives} (or the {\it free partial  derivatives}). Every $d_j$ is well defined via $\bar{\mu }$ on  $\mathbb{Z}\bar{G}.$ 

An arbitrary element  $u$ in $\bar{G}$  is trivial if and only if $d_i(u) = 0$ for all $i = 1, ..., r.$
Thus  two elements $u, v$  are equal in $\bar{G}$ if and only if $d_i(u) = d_i(v)$ for all $i = 1, ..., r.$

The main identity for the Fox derivatives (\ref{eq:mainf})  gives a similar one in the considered case too:

\begin{equation}
 \label{eq:maingen}
\sum_{i=1}^{r} d_i(\alpha )(\mu (f_i) - 1) = \bar{\mu}(\alpha )  - \varepsilon (\bar{\mu} (\alpha)),
\end{equation}

\noindent
where $\alpha \in \mathbb{Z}\bar{G}$ and $\varepsilon : \mathbb{Z}G \rightarrow \mathbb{Z}$ is the trivialization homomorphism  on $\mathbb{Z}G.$

Let $u \in \bar{R}$ and $\alpha \in \mathbb{Z}G.$ Then
\begin{equation}
 \label{eq:derualpha}
d_i(u^{\alpha }) = \alpha d_i(u) \  \textrm{for all} \  i = 1, ..., r.
\end{equation}

\subsection{Matrix representations}

\paragraph{Magnus representations for automorphism groups.} 
\label{ss:mrae}

Let $F_r$ be a free group with basis $\{f_1, ..., f_r\}.$ 

\begin{definition}
 \label{Magnusrepraut}
The {\it Magnus representation} for Aut$(F_r)$ is the map 
\begin{equation}
 \label{eq:Mra}
\alpha : \mathrm{Aut}(F_r) \rightarrow GL_r(\mathbb{Z}F_r)
\end{equation}

\noindent
assigning to $\varphi \in \mathrm{Aut}(F_r)$ the Jacobi matrix 

\begin{equation}
 \label{eq:jm}
J(\varphi ) = (D_j(\varphi (f_i))
\end{equation}

\noindent
where $D_j$ are the free Fox derivatives for $j = 1, ..., r,$ and   $D_j(\varphi (f_i))$ is the $ij-$th entry of the matrix.
\end{definition}

The Magnus representation $\alpha $ for Aut$(F_r)$ is injective since $\varphi (f_i)$ is recovered from $J(\varphi )$ by applying the main identity for free Fox derivatives  to the $i-$th row for each $\varphi \in \mathrm{Aut}(F_r).$

The Magnus representation is not homomorphism as is seen from the following assertion. 

\begin{proposition}
 \label{eq:tw}
For $\varphi , \psi \in \mathrm{Aut}(F_r),$ the equality

\begin{equation}
 \label{eq:semi}
J(\varphi \psi ) = \psi (J(\varphi )) \cdot J(\psi ) 
\end{equation}
\noindent
holds, where $\psi (J(\varphi ))$ means that $\psi $ is applied to every entry in $J(\varphi ).$ 
\end{proposition}

In particular, it follows that  the image of $\alpha $ is contained in the group $GL_r(\mathbb{Z}F_r)$ of invertible matrices.

To obtain a genuine representation, a homomorphism, we need  to change $F_r$ to a group $\bar{G}$ of type $F_r/R',$ where $R$ is a normal subgroup in $F_r.$ Let $G = F_r/R$ and $\bar{G} = F_r/R'.$  Then $\bar{R} = R/R'$ is normal abelian subgroup of $\bar{G},$ which is considered as a module over the group ring $\mathbb{Z}G.$ Let $\{x_1, ..., x_r\}$ be the generators of $\bar{G}$ corresponding to the basic elements $\{f_1, ..., f_r\}$ of $F_r.$ 

\begin{definition}
 \label{reprgen}
The {\it Magnus representation for} Aut$(\bar{G})$ is the map 
\begin{equation}
 \label{eq:Mra}
\alpha_A : \mathrm{Aut}(\bar{G}) \rightarrow GL_r(\mathbb{Z}G)
\end{equation}

\noindent
assigning to $\varphi \in \mathrm{Aut}(\bar{G})$ the Jacobi matrix $J(\varphi ) = (d_j(\varphi (x_i) )$ of $\varphi $ 
over $\mathbb{Z}G.$ Here $d_j$  are the induced free Fox derivatives on $\bar{G}$ with values in $\mathbb{Z}G$ with respect to generators $x_i$ for $i, j = 1, ..., r,$  and $d_j(\varphi (x_i))$ is the $ij-$th entry in the matrix.
\end{definition}

For $R$ a normal subgroup of a group $H,$ let IRAut$(H)$ denote the group of all automorphisms of  $H$ for which $R$ is invariant, 
and which induce the identical map on the quotient $H/R.$

Then the map  $\alpha_{AR} : \mathrm{IRAut}(\bar{G}) \rightarrow GL_r(\mathbb{Z}G)$  induced by the map $\alpha_A $ is   homomorphism.

Let $M_r$ be the free metabelian group of rank $r$ with basis $\{x_1, ..., x_r\},$  and $A_r = M_r/M_r'$ be the abealization of $M_r$ with the corresponding basis $\{a_1, ..., a_r\}.$
The group ring $\mathbb{Z}A_r$ can be considered as a Laurent polynomial ring $\Lambda_r = \mathbb{Z}[a_1^{\pm 1}, ..., a_r^{\pm 1}]. $  Let $\psi $ be an automorphism of $M_r.$ 
We define the Jacobi matrix $J( \psi )$ corresponding to $\psi $ as above.
Then the map
\begin{equation}
 \beta : \psi \rightarrow J(\psi )
\end{equation}
\noindent
gives an injective homomorphism (embedding) of $\mathrm{IAut}(M_r)$ into the group of matrices of size $r$ over the Laurent polynomial ring 
$\Lambda_r .$ This  homomorphism is called {\it Bachmuth's embedding}.

Every automorphism $\varphi \in $ Aut$(F_r)$ induces an automorphism $\bar{\varphi } \in $ Aut$(M_r).$ Let $r = 3.$ We compute the matrices $J(\bar{\xi}_{1,2}), J(\bar{ \xi}_{1,3})$ and $J(\bar{\rho}_{1,2,3})$ where $\bar{\xi}_{1,3}, \bar{\xi}_{2,3}, \bar{\rho}_{1,2,3}$ are induced by $\xi_{1,3}, \xi_{2,3}, \rho_{1,2,3} \in $ IAut$(F_3),$ respectively.

\begin{equation}
\label{eq:4.6}
J(\bar{\xi}_{1,2}) = \left( \begin{array}{ccc}
a_2 & 1-a_1 & 0\\
0&1&0\\
0&0&1\\
\end{array}\right),
\end{equation}

\begin{equation}
\label{eq:4.7}
J(\bar{\xi}_{1,3}) = \left( \begin{array}{ccc}
a_3 &  0  & 1-a_1 \\
0&1&0\\
0&0&1\\
\end{array}\right),
\end{equation}

\begin{equation}
\label{eq:4.8}
J(\bar{\rho}_{1,2,3}) = \left( \begin{array}{ccc}
1 &a_1( 1-a_3) & a_1(a_2-1)\\
0&1&0\\
0&0&1\\
\end{array}\right).
\end{equation}

Let $\eta : \Lambda_3 \rightarrow \mathbb{Z}[t, t^{-1}]$ be a homomorphism of the Laurent rings defined by the map 

\begin{equation}
\label{eq:4.9}
\eta : a_i \mapsto 1 \  \mathrm{for} \  i = 1,2; a_3 \mapsto t.
\end{equation}

Now we can define by ( \ref{eq:4.6} - \ref{eq:4.9}) a homomorphism $\bar{\eta} : $ IAut$(F_3) \rightarrow GL_2(\mathbb{Z}[t, t^{-1}])$ such that

\begin{equation}
\label{eq:10}
\bar{\eta}(\xi_{1,2}) = 1, \bar{\eta}(\xi_{1,3}) = \left( \begin{array}{cc}
t & 0\\
0&1\\
\end{array}\right),  \bar{\eta}(\rho_{1,2,3}) = \left(\begin{array}{cc}
1& 1 -t\\
0&1\\
\end{array}\right).
\end{equation}

The following lemma, which has been proved in \cite{KrMc}, provides essentially infinite sets of relations of $PGL_2(\mathbb{Z}[t, t^{-1}]).$ We write $d, e_0, u$ for the matrices

\begin{equation}
d =  \left( \begin{array}{cc}
t & 0\\
0&1\\
\end{array}\right), e_0 = \left( \begin{array}{cc}
1 & 1\\
0&1\\
\end{array}\right), u = \left( \begin{array}{cc}
1 & t-1\\
0&1\\
\end{array}\right).
\end{equation}

\begin{lemma}
\label{le:2}
 \cite{KrMc}. Let $K$ be a field and let $Q$ be any infinite subset of 

\begin{equation}
\label{eq:Q_1}
Q_1 = \{[e_0, d^ke_0d^{-k}] | k \geq 1\}
\end{equation}

\noindent
or

\begin{equation}
Q_2 = \{ d^k(ud)^{-k}d^k(u^{-1}d)^{-k} | k \geq 1 \}.
\end{equation}

Then $Q$ is an essentially infinite set of relations in $PGL_2(K[t, t^{-1}]).$
\end{lemma}

Lemma \ref{le:2} has been used in\cite{KrMc} in establishing Theorem \ref{th:1.1}. We will use it in establishing our main result on non-finite presentability of the generalized Andrews-Curtis groups.

\section{Proof  of Theorem B}
\label{sec:5}

The homomorphism $\bar{\eta}$ defined by (\ref{eq:10})  can be naturally extended via the Magnus representation and setting $\xi_{i,j} \mapsto 1$ for $i = 1,2; j = 2, ..., m$ to a homomorphism 

\begin{equation}
\nu : A_{2,m}  \rightarrow PGL_2(\mathbb{Z}[t, t^{-1}]).
\end{equation}

Let $\lambda_{1,2} $ be an automorphism in $A_{2,m}$ which maps $x_1 \mapsto x_2x_1$ and fixes other generators. Also let $\rho_{1,2}$ be an automorphism in $A_{2,m}$ which maps $x_1 \mapsto x_1x_2$ and fixes other generators. Then $\lambda_{1,2}$ commutes with all automorphisms of the form $\xi_{2,3}^k\rho_{1,2}\xi_{2,3}^{-k}$ for every integer $k.$ 

By direct computation we get that 

\begin{equation}
\nu (\lambda_{1,2}) = e_0, \nu (\xi_{2,3}^k\rho_{1,2}\xi_{2,3}^{-k}) = d^ke_0d^{-k} \  \textrm{for every integer } \  k.
\end{equation}

Since the set $Q_1$ in (\ref{eq:Q_1})  is essentially infinite the set of relations $\{[\lambda_{1,2},\xi_{2,3}^k\rho_{1,2}\xi_{2,3}^{-k}] | k \in \mathbb{N}\}$ which maps onto $Q_1$ by $\nu $
is essentially infinite in $A_{2,m}.$ Hence the group $A_{2,m}$ is non-finitely presented.
 
Therefore Theorem B is proved. 

\medskip
\noindent
{\bf Remark} {\it
For any $m,$  $A_{1, m}$ is finite presentable, and for any $n,$ $A_{n,0}$ is finite presentable. 
}

Indeed,  for any $m,$  $A_{1, m}$ is isomorphic to  a direct product of groups $C_2 \times F_m,$ where $C_2 \simeq  \  \textrm{Aut}(F_1)$ is the cyclic group of order $2$ and $F_m=$
gp$(\xi_{1,1}, ..., \xi_{1,m})$  is the free group with basis $\xi_{1,1}, ..., \xi_{1,m}.$ Hence every group $A_{1, m}$ is a finitely presentable  group.

 \section{Open problems}
 \label{sec:6}

 In this section we state several open problems on Andrews-Curtis groups.
 
 \begin{problem}
 Are any of the groups $AC(F_r)$, $r >2$, finitely presentable?
 
 \end{problem}

 \begin{problem}
 For which normally $r$-generated groups $G$ the canonical epimorphism $\phi: GAC_r(G) \to AC(G)$ is an isomorphism? 
 In particular:
 \begin{itemize}
 \item Is $\phi: GAC_r(G) \to AC(G)$ an isomorphism for a torsion-free non-elementary hyperbolic group $G$?
 \item For which  $r$-generated partially commutative groups $G = G(\Gamma)$ the canonical epimorphism $\phi: GAC_r(G) \to AC(G)$ is an isomorphism?
 
 \end{itemize}
 \end{problem}

 \begin{problem}
 Find "good" (quasi-geodesic) normal forms of elements in $AC(F_r)$
 \end{problem}
 
 Solution to this problem will enhance efficacy of computations with ACC.
 
 \begin{problem}
 For which $r$ does the group $AC(F_r)$ have Kazhdan property (T)?
 \end{problem}

Aknowledgments.

The author is grateful to A. Myasnikov for invaluable help in preparing the paper.

\bigskip

 Author's information

Vitaly Roman'kov: 

Sobolev Institute of Mathematics (Omsk Branch), Pevtsova str. 13, 644099, Omsk, Russia.

E-mail: romankov48@mail.ru

\end{document}